\documentclass[12pt]{amsart}
\usepackage{amssymb,amsmath}
\usepackage{mathtools}  
\usepackage{graphicx}
\usepackage{bm,bbm}
\usepackage{color}
\usepackage{pgfplots}
\usepackage[hidelinks]{hyperref}
\newtheorem{theorem}{Theorem}
\newtheorem{lemma}[theorem]{Lemma}
\newtheorem{remark}[theorem]{Remark}
\newcommand{\cH}{\ensuremath{\mathcal{H}}}
\newcommand{\V}{\ensuremath{\mathcal{V}}}
\newcommand{\Q}{\ensuremath{\mathcal{Q}} }
\newcommand{\uo}{\ensuremath{u_0}}
\newcommand{\po}{\ensuremath{p_0}}

\def\R{\mathbb{R}}

\definecolor{myBlue1}{RGB}{101,149,239}  
\definecolor{myBlue2}{RGB}{113,104,238} 
\definecolor{myBlue3}{RGB}{30,144,255} 
\definecolor{myGreen1}{RGB}{154,204,50} 
\definecolor{myGreen2}{RGB}{69,169,0} 
\definecolor{myGreen3}{RGB}{154,205,50} 
\definecolor{myGreen4}{RGB}{105,139,34} 
\definecolor{myRed1}{RGB}{210,105,30} 
\definecolor{myRed2}{RGB}{165,42,42} 
\definecolor{myRed3}{RGB}{139,26,26} 
\definecolor{lightgray}{RGB}{175,175,175} 
\definecolor{myLGray}{RGB}{225,225,225} 
\DeclareMathOperator{\diag}{diag}

\newcommand{\calB}{\ensuremath{\mathcal{B}} }

\newcommand{\calD}{\ensuremath{\mathcal{D}} }

\newcommand{\calK}{\ensuremath{\mathcal{K}} }

\def\dt{\,\text{d}t}
\def\dx{\,\text{d}x}

\newcommand{\partialkn}{\ensuremath{\partial_{\kappa,n}} }
\newcommand{\hook}{\hookrightarrow}
\begin{document}
\title[A PDAE Formulation of Dynamic Boundary Conditions]{A PDAE Formulation of Parabolic Problems with\\ Dynamic Boundary Conditions}
\author{R. Altmann}
\email{robert.altmann@math.uni-augsburg.de}
\address{Department of Mathematics, University of Augsburg, Universit\"atsstr.~14, 86159 Augsburg, Germany}
\date{\today}
%
%
\begin{abstract}
The weak formulation of parabolic problems with dynamic boundary conditions is rewritten in form of a partial differential-algebraic equation. More precisely, we consider two dynamic equations with a coupling condition on the boundary. This constraint is included explicitly as an additional equation and incorporated with the help of a Lagrange multiplier. Well-posedness of the formulation is shown. 
\end{abstract}
%
%
\maketitle
%
%
%
\section{Introduction}
Within this paper, we present a new abstract formulation of parabolic initial-boundary value problems with dynamic boundary conditions. This includes locally reacting {\em Wentzell boundary conditions} as well as non-local {\em kinetic boundary conditions} modeling a diffusion on the surface of the computational domain. In both cases, we derive a formulation as partial differential-algebraic equation (PDAE) where the boundary condition is treated in form of a coupling condition. This constraint is included to the system as an additional equation and enforced with the help of a Lagrange multiplier. The resulting structure is then similar to the one we would obtain for a parabolic problem with pure Dirichlet boundary conditions. \smallskip

In general, dynamic boundary conditions appear in applications where the momentum on the boundary should not be neglected. For hyperbolic systems, such boundary conditions are of particular interest in fluid-structure interaction, where one component can be modelled in form of a boundary layer instead of using a full model~\cite{Hip17}. For the heat equation, dynamic boundary conditions can be found in~\cite[Ch.~2]{DuvL76} and enable a proper way to model a heat source on the boundary~\cite{Gol06}. On the other hand, dynamic boundary conditions may be used as stabilizing feedback control on the boundary~\cite{KomZ90}. 

An abstract framework for the weak formulation of parabolic systems with dynamic boundary conditions was recently presented in~\cite{KovL17}. Therein, it was shown that the problem fits into the standard formulation of parabolic problems if the inner products are adapted accordingly. 
Formulations including semigroups have received much more attention in the literature. 
Here, the boundary conditions are prescribed within the domain of the differential operator. The proof of the analytic semigroup property in~\cite{EngF05} considers a decoupling similar to the idea presented in the present paper but without the inclusion of a Lagrange multiplier.  \smallskip

In this paper, we interpret such parabolic initial-boundary value problems as a coupled system including dynamics in the bulk and on the boundary. As we consider the weak formulation of the problem, we have a direct connection to spatial discretization schemes. This may lead to a new class of robust numerical methods and allows independent meshes in the domain and on the boundary.
%
%
\section{Preliminaries}\label{sec:prelim}
This section is devoted to the introduction of needed functions spaces and the trace operator. Further, we discuss parabolic systems with Dirichlet boundary conditions and their formulation as constrained system. It turns out that this is helpful for the formulation presented in Section~\ref{sec:dynBC}.
%
%
\subsection{Function spaces and the trace operator}\label{sec:prelim:spaces}
Throughout this paper, $\Omega \subset \R^d$ denotes a bounded Lipschitz domain with boundary $\Gamma:= \partial\Omega$ and $[0, T]$ the bounded time interval of interest. We consider the standard Sobolev spaces
\[
  V := H^1(\Omega), \qquad
  V_0 := H^1_0(\Omega), \qquad 
  Q := H^{-1/2}(\Gamma). 
\]
The corresponding dual spaces are given by~$V^*$,~$V_0^*=H^{-1}(\Omega)$, and~$Q^*= H^{1/2}(\Gamma)$, respectively. With the pivot space~$H:=L^2(\Omega)$ this leads to the Gelfand triples $V, H, V^*$ and $V_0, H, V^*_0$, see~\cite[Ch.~23.4]{Zei90a}. Also the trace spaces form a Gelfand triple, namely with the pivot space $L^2(\Gamma)$. 

In the bulk as well as on the boundary, we write $(\, \cdot\, , \cdot\, )$ for the inner product of the corresponding pivot space, i.e., for $u, v \in H$ and $p, q \in L^2(\Gamma)$ we have 
\[
  (u, v) 
  := (u, v)_{L^2(\Omega)}
  = \int_\Omega u(x)\, v(x) \dx, \qquad
  (p, q) 
  := (p, q)_{L^2(\Gamma)}
  = \int_\Gamma p(x)\, q(x) \dx.
\]
Further, we use duality pairings, i.e., for $u\in V$ and $f\in V^*$ we write $\langle f, u\rangle := \langle f, u\rangle_{V^*,V}$. If we have~$f\in L^2(\Omega)$, then the embedding given by the Gelfand triple implies~$\langle f, u\rangle = (f, u)$. Since we deal with time-dependent problems, the solution spaces will be Sobolev-Bochner spaces, cf.~\cite[Ch.~25]{Wlo87}. This includes spaces of the form~$L^2(0,T;V)$, i.e, abstract functions~$u\colon [0,T] \to V$ with $\int_0^T \Vert u(t)\Vert^2_V \dt < \infty$. 

A decisive role in this paper plays the trace operator, which we denote by~$\calD\colon V \to Q^*$. Recall that this operator is defined through~$\calD u := u|_\Gamma$ for continuous functions $u \in C^0(\overline \Omega) \cap V$, which are dense in~$V$. For prescribed boundary data~$g\in Q^*$, Dirichlet boundary conditions then simply read~$\calD u = g$. The trace operator~$\calD$ is surjective~\cite[Th.~2.21]{Ste08} and satisfies an inf-sup condition \cite[Lem.~4.7]{Ste08}. The corresponding dual operator~$\calD^*$ maps from~$Q$ to~$V^*$ and is defined by~$\langle \calD^* \lambda, u \rangle := \langle \lambda, \calD u \rangle_{Q,Q^*}$ for all~$u\in V$ and~$\lambda \in Q$.
%
%
\subsection{Homogeneous Dirichlet boundary conditions}\label{sec:prelim:homDirichlet}
In order to get used to the abstract setting, we first consider the homogeneous case, i.e., we consider the parabolic equation   
\begin{align}
\label{eq:homDirichletBC}
  \dot u - \nabla\cdot(\kappa \nabla u) = f \quad\text{in } \Omega, \qquad
  u = 0 \quad\text{on } \Gamma
\end{align}
with initial condition $u(0) = \uo$. For the diffusion parameter~$\kappa \in L^\infty(\Omega)$ we assume~$\kappa(x) \ge c_\kappa > 0$ for almost all~$x\in \Omega$. The corresponding weak formulation, which we obtain by an application of the integration by parts formula, can be written as an operator equation, namely 
\begin{align}
\label{eq:PDAE:homDirichlet}
  \dot u + \calK u = f \qquad \text{in } V_0^*.
\end{align}
This equation is stated in~$V_0^*$, meaning that we consider test functions in $V_0$. The linear operator~$\calK\colon V\to V^*$ is defined for~$u, v\in V$ by
\begin{align}
\label{def:calK}
  \langle \calK u, v\rangle
  := \int_{\Omega} (\kappa \nabla u) \cdot \nabla v  \dx.
\end{align}
This operator is continuous and, due to the assumptions on~$\kappa$, elliptic on the subspace~$V_0$. We emphasize that we may replace~$\calK$ by any other continuous operator mapping from~$V$ to~$V^*$ and satisfying a G\aa{}rding inequality on~$V_0$. This then leads to the well-known result that~$f\in L^2(0,T;V_0^*)$ and~$\uo \in H$ imply the existence of a unique solution of~\eqref{eq:PDAE:homDirichlet} satisfying~$u \in L^2(0,T;V_0)$ and~$\dot u \in L^2(0,T;V_0^*)$, cf.~\cite[Ch.~26]{Wlo87}. 
%
%
\subsection{Inhomogeneous Dirichlet boundary conditions}\label{sec:prelim:Dirichlet}
We now turn to general time-dependent Dirichlet boundary conditions. For this, we replace the homogeneous boundary condition in~\eqref{eq:homDirichletBC} by~$u = g$ on~$\Gamma$ for given data~$g\colon [0,T] \to Q^*$. 
%
The standard procedure is to construct a function~$u_g$ on the domain~$\Omega$ with~${u_g}|_\Gamma = g$ in each time point. Then, one considers~$\tilde u := u - u_g$ as the new unknown. Assuming sufficient regularity, this leads to an operator equation for~$\tilde u$ similar to~\eqref{eq:PDAE:homDirichlet} but with an adapted right-hand side and initial condition.

As preparation for the following section, we rather consider the formulation as constrained operator equation, where the boundary condition is given as an explicit constraint, cf.~\cite{Sim00,Alt15}. This then leads to a PDAE. In contrast to the homogeneous case, we consider test functions in~$V$ rather than~$V_0$ and introduce a Lagrange multiplier~$\lambda\colon [0,T] \to Q$ in order to enforce the constraint~$\calD u = g$. This then leads to the constrained system 
\begin{subequations}
\label{eq:PDAE:Dirichlet}
\begin{align}
	\dot u + \calK u + \calD^*\lambda &= f\qquad\text{in } V^*, \label{eq:PDAE:Dirichlet:a} \\
	\calD u \phantom{j + \calD\lambda} &= g\qquad\text{in } Q^*.  \label{eq:PDAE:Dirichlet:b}
\end{align}
\end{subequations}
The solution $u$ of~\eqref{eq:PDAE:Dirichlet} obviously satisfies the boundary conditions. Further, restricting the test functions in~\eqref{eq:PDAE:Dirichlet:a} to $V_0$, we obtain~$\dot u + \calK u = f$. In this particular case, the Lagrange multiplier allows the physical interpretation of the normal derivative of $u$ along the boundary~\cite{Sim00}. 
%
%
\begin{theorem}
Assume right-hand sides~$f\in L^2(0,T;V^*)$, $g\in H^1(0,T;Q^*)$ and a consistent initial condition $u(0) = \uo \in V$, i.e., $\calD \uo = g(0)$. Then, system~\eqref{eq:PDAE:Dirichlet} has a unique distributional solution satisfying~$u \in L^2(0,T;V)$ and~$\dot u + \calD^* \lambda \in L^2(0,T;V^*)$.  
\end{theorem}
\begin{proof}
The result follows by the positivity of~$\calK$ and the inf-sup stability of~$\calD$, cf.~\cite[Th.~3.3]{EmmM13}. 
\end{proof}
%
%
\section{Dynamic boundary conditions}\label{sec:dynBC}
Finally, we address linear parabolic problems with dynamic boundary conditions, i.e., 
\begin{subequations}
\label{eq:dynamicBC}
\begin{align}
  \dot  u - \nabla\cdot(\kappa \nabla u) &= f \qquad\text{in } \Omega, \label{eq:dynamicBC:a} \\
  \dot u - \beta\, \Delta_\Gamma u + \partialkn u + \alpha\, u &= g \qquad \text{on } \Gamma \label{eq:dynamicBC:b}
\end{align}
\end{subequations}
with initial condition $u(0) = \uo$. We denote the unit normal vector by~$n$ and the corresponding normal derivative by~$\partial_{\kappa,n} u := n\cdot(\kappa \nabla u)$. The Laplace-Beltrami operator is denoted by~$\Delta_\Gamma$, see~\cite[Ch.~16.1]{GilT01}. 
For the parameters we assume~$\kappa$ to be positive as before, $\alpha \in L^\infty(\Omega)$, and~$\beta\ge 0$ constant. The boundary condition~\eqref{eq:dynamicBC:b} is called \emph{locally reacting} for $\beta= 0$ and \emph{non-local} otherwise. The aim of this section is to derive an operator formulation of~\eqref{eq:dynamicBC} as a coupled system. 
%
%
\subsection{Locally reacting dynamic boundary conditions}\label{sec:dynBC:local}
We consider dynamic boundary conditions without the Laplace-Beltrami operator, i.e., we assume $\beta= 0$. These so-called \emph{Wentzell boundary conditions} then read 
\[
  \dot u + \partialkn u + \alpha\, u 
  = g \qquad \text{on } \Gamma. 
\]
We now introduce a new variable, namely $p:=u|_\Gamma$, such that system~\eqref{eq:dynamicBC} with $\beta=0$ is equivalent to 
\begin{subequations}
\label{eq:dynBC:local:newVar}
\begin{align}
  \dot  u - \nabla\cdot(\kappa \nabla u) &= f \qquad\text{in } \Omega, \label{eq:dynBC:local:newVar:a} \\
  \dot p + \partialkn u + \alpha\, p &= g \qquad \text{on } \Gamma,  \label{eq:dynBC:local:newVar:b} \\
  p - u &= 0 \qquad\text{on } \Gamma. \label{eq:dynBC:local:newVar:c} 
\end{align}
\end{subequations}
Note that~\eqref{eq:dynBC:local:newVar:a} and~\eqref{eq:dynBC:local:newVar:b} are dynamic equations for~$u$ and~$p$, respectively, whereas equation~\eqref{eq:dynBC:local:newVar:c} marks the coupling condition. For the weak formulation we define the spaces
\[
  \V := H^1(\Omega) \times H^{1/2}(\Gamma), \qquad
  \cH := L^2(\Omega) \times L^{2}(\Gamma), \qquad  
  \Q := H^{-1/2}(\Gamma).
\]
On~$\V$ and~$\cH$ we define the norms~$\Vert \big[ \substack{u\\ p} \big] \Vert^2_\V := \Vert u \Vert^2_{H^1(\Omega)} + \Vert p \Vert^2_{H^{1/2}(\Gamma)}$ and $\Vert \big[ \substack{u\\ p} \big] \Vert^2_\cH := \Vert u \Vert^2_{L^2(\Omega)} + \Vert p \Vert^2_{L^{2}(\Gamma)}$, respectively. Further, we define the coupling operator $\calB\colon \V \to \Q^*$ as
\[
  \calB\, \big[ \substack{u\\ p} \big]
  := p - \calD u \in \Q^*.
\]
The dual operator $\calB^*\colon \Q \to \V^*$ is defined accordingly via~$\langle \calB^* \lambda, \big[ \substack{u\\ p} \big] \rangle := \langle \lambda, p - \calD u \rangle_{\Q, \Q^*}$. We now consider a test function $\big[ \substack{v\\ q} \big] \in\V$ and multiply equation~\eqref{eq:dynBC:local:newVar:a} by~$v$ and~\eqref{eq:dynBC:local:newVar:b} by~$q$. With~$\calK$ as defined in~\eqref{def:calK}, integration by parts and taking the sum leads to 
\begin{align*}
  \langle f, v\rangle + (g, q)
  &= (\dot u, v) + \langle \calK u, v \rangle - \langle \partialkn u, \calD v \rangle 
  + (\dot p, q) + \langle \partialkn u, q \rangle + (\alpha p,q)  \\
  &= (\dot u, v) + (\dot p, q) + \langle \calK u, v \rangle  
  + \langle \partialkn u, \calB \big[ \substack{v\\ q} \big]\, \rangle + (\alpha p,q).
\end{align*}
Proceeding similarly as in Section~\ref{sec:prelim:Dirichlet}, we introduce a Lagrange multiplier $\lambda\colon [0,T] \to \Q$ in place of the normal trace $\partialkn u$. This then leads to the PDAE  
\begin{subequations}
\label{eq:PDAE:local}
\begin{align}
  \begin{bmatrix} \dot u \\ \dot p  \end{bmatrix}
  + \begin{bmatrix} \calK &  \\  & \alpha \end{bmatrix}
  \begin{bmatrix} u \\ p  \end{bmatrix}
  + \calB^*\lambda 
  &= \begin{bmatrix} f \\ g \end{bmatrix} \qquad \text{in } \V^*, \\
  \calB\, \begin{bmatrix} u \\ p  \end{bmatrix} \phantom{i + \calB \lambda}
  &= \phantom{[]} 0\hspace{2.751em} \text{in } \Q^*. 
\end{align}
\end{subequations}
%
%
with initial conditions~$u(0) = \uo$ and~$p(0) = \po$. Note that the structure of this operator equation is similar to the pure Dirichlet case in~\eqref{eq:PDAE:Dirichlet}. For the well-posedness one needs to discuss the inf-sup stability of~$\calB$ as well as the positivity of the differential operator~$\diag(\calK, \alpha)$.
\begin{lemma}
\label{lem:infsup:local}
The coupling operator~$\calB\colon \V \to \Q^*$ is inf-sup stable with constant~$1$, i.e.,
\[
  \adjustlimits \inf_{q\in \Q} \sup_{\big[ \substack{u\\ p} \big] \in \V} 
  \frac{\big\langle \calB \big[ \substack{u\\ p} \big],\, q \big\rangle_{\Q^*, \Q}}{\Vert \big[ \substack{u\\ p} \big] \Vert_\V\, \Vert q\Vert_\Q}
  \ge 1.
\] 
Further, the operator~$\diag(\calK, \alpha)\colon \V\to\V^*$ is continuous and satisfies a G\aa{}rding inequality on~$\V_0 := \ker \calB$. 
\end{lemma}
\begin{proof}
For the inf-sup	stability let~$q\in \Q = H^{-1/2}(\Gamma)$ be arbitrary with norm~$1$. We set $\tilde u = 0$ and $\tilde p$ as the Riesz representation of $q$ in the Hilbert space $\Q^*= H^{1/2}(\Gamma)$ such that 
\[
  \sup_{\big[ \substack{u\\ p} \big] \in \V} 
  \frac{\big\langle \calB \big[ \substack{u\\ p} \big],\, q \big\rangle}{\Vert \big[ \substack{u\\ p} \big] \Vert_\V }
  \ge \frac{\langle \tilde p - \calD \tilde u,\, q \rangle}{\Vert \big[ \substack{\tilde u\\ \tilde p} \big] \Vert_\V}
  = \frac{\langle \tilde p,\, q \rangle}{\Vert \tilde p \Vert_{H^{1/2}(\Gamma)}}
  = \frac{\Vert \tilde p \Vert^2_{\Q^*}}{\Vert \tilde p \Vert_{\Q^*}}
  = \Vert \tilde p \Vert_{\Q^*}
  = \Vert q \Vert_{\Q}
  = 1.   
\]
The continuity of~$\diag(\calK, \alpha)$ follows directly form the continuity of~$\calK$ and the continuity of the embedding~$L^{2}(\Gamma) \hook H^{-1/2}(\Gamma)$. For the G\aa{}rding inequality, we consider~$\big[\substack{u\\ p} \big] \in \V_0 = \{ \big[ \substack{u\\ p} \big] \in \V\, |\, \calD u = p \}$ and estimate with the continuity constant of the trace operator~$C_\text{tr}$ and $C_\alpha := \Vert \alpha\Vert_{L^\infty(\Omega)}$, 
\begin{align*}
  \Big\langle \begin{bmatrix} \calK &  \\  & \alpha \end{bmatrix} 
  \begin{bmatrix} u \\ p \end{bmatrix},  \begin{bmatrix} u \\ p \end{bmatrix} \Big\rangle
  &\ge c_\kappa \Vert u \Vert^2_{H^1(\Omega)} - c_\kappa\Vert u \Vert^2_{L^2(\Omega)} + (\alpha p, p) \\
  &\ge \frac{c_\kappa}{2} \Vert u \Vert^2_{H^1(\Omega)} + \frac{c_\kappa}{2 C^2_\text{tr}} \Vert p \Vert^2_{H^{1/2}(\Gamma)} - c_\kappa\Vert u \Vert^2_{L^2(\Omega)} - C_\alpha \Vert p \Vert^2_{L^2(\Gamma)}  \\
  &\gtrsim \Vert \big[ \substack{u\\ p} \big] \Vert^2_{\V} - c\, \Vert \big[ \substack{u\\ p} \big] \Vert^2_{\cH}.
\end{align*}
Note that we have applied the fact that~$u$ and~$p$ can be replaced by each other on the boundary. 
\end{proof}
Using again~\cite[Th.~3.3]{EmmM13}, we obtain by the previous lemma the following well-posedness result. 
\begin{theorem}
Given a bounded Lipschitz domain~$\Omega$, we assume~$\big[ \substack{f\\ g} \big] \in L^2(0,T;\V^*)$ and consistent initial data~$\big[ \substack{\uo\\ \po} \big] \in \V_0$. Then, system~\eqref{eq:PDAE:local} has a unique distributional solution~$\big[ \substack{u\\ p} \big] \in L^2(0,T; \V)$, i.e, $u \in L^2(0,T; H^1(\Omega))$~and~$p \in L^2(0,T; H^{1/2}(\Gamma))$. 
\end{theorem}
\begin{remark}
We emphasize that the condition on the initial value may be relaxed to~$\big[ \substack{\uo\\ \po} \big] \in \cH$, cf.~\cite[Rem.~6.9]{Alt15}. Further, we like to mention the connection to the abstract setting introduced in~\cite{KovL17}. Therein, the ansatz space equals~$\V_0=\ker\calB$, i.e., the connection of~$u$ and the boundary variable~$p$ is given a priori.
\end{remark}
%
%
%
\subsection{Nonlocal dynamic boundary conditions}\label{sec:dynBC:nonlocal}
Finally, we consider the case~$\beta > 0$. Due to the Laplace-Beltrami operator the boundary condition is then non-local. As in the previous subsection, we introduce~$p:=u|_\Gamma$ such that~\eqref{eq:dynamicBC} is equivalent to 
\begin{subequations}
\label{eq:dynBC:nonLocal:newVar}
\begin{align}
	\dot  u - \nabla\cdot(\kappa \nabla u) &= f \qquad\text{in } \Omega, \label{eq:dynBC:nonLocal:newVar:a} \\
	\dot p - \beta\, \Delta_\Gamma p + \partialkn u + \alpha\, p &= g \qquad \text{on } \Gamma,  \label{eq:dynBC:nonLocal:newVar:b} \\
	p - u &= 0 \qquad\text{on } \Gamma. \label{eq:dynBC:nonLocal:newVar:c} 
\end{align}
\end{subequations}
The presence of the Laplace-Beltrami operator requires an adjustment of the ansatz spaces to 
\[
  \V := H^{1}(\Omega) \times H^{1}(\Gamma), \qquad
  \cH := L^2(\Omega) \times L^{2}(\Gamma), \qquad  
  \Q := H^{-1/2}(\Gamma)
\]
with the corresponding norm $\Vert \big[ \substack{u\\ p} \big] \Vert^2_\V := \Vert u \Vert^2_{H^1(\Omega)} + \Vert p \Vert^2_{H^{1}(\Gamma)}$. The coupling operator $\calB\colon \V \to \Q^*$ remains unchanged. 
In order to obtain an operator formulation of~\eqref{eq:dynBC:nonLocal:newVar} we again derive the weak formulation via integration by parts. Here, we also have to integrate by parts on the boundary leading to the weak form of the Laplace-Beltrami operator, which we denote by~$\calK_\Gamma\colon H^{1}(\Gamma) \to H^{-1}(\Gamma)$. With a Lagrange multiplier~$\lambda\colon [0,T] \to \Q$ this then leads to  
\begin{subequations}
\label{eq:PDAE:nonlocal}
\begin{align}
	\begin{bmatrix} \dot u \\ \dot p  \end{bmatrix}
	+ \begin{bmatrix} \calK &  \\  & (\calK_\Gamma + \alpha) \end{bmatrix}
	\begin{bmatrix} u \\ p  \end{bmatrix}
	+ \calB^*\lambda 
	&= \begin{bmatrix} f \\ g \end{bmatrix} \qquad \text{in } \V^*, \\
	\calB\, \begin{bmatrix} u \\ p  \end{bmatrix} \phantom{i + \calB \lambda} &= \phantom{[]} 0\hspace{2.751em} \text{in } \Q^*.
\end{align}
\end{subequations}
%
%
For the well-posedness, we consider the following lemma.
\begin{lemma}
\label{lem:infsup:nonLocal}
Let~$C_\text{invTr}$ denote the continuity constant of the inverse trace theorem, cf.~\cite[Th.~2.22]{Ste08}. Then, the operator $\calB\colon \V \to \Q$ is inf-sup stable with constant~$1/C_\text{invTr}$. 
Further, the differential operator~$\diag(\calK, \calK_\Gamma+\alpha)\colon \V\to\V^*$ is continuous and satisfies a G\aa{}rding inequality on~$\V$. 
\end{lemma}
\begin{proof}
Let $q\in \Q$ be arbitrary with norm~$1$ and $\tilde q \in \Q^*$ its Riesz representation. Further, set~$\tilde p = 0$ and $\tilde u \in V$ as the weak solution of $-\Delta \tilde u = 0$ with Dirichlet boundary condition $\calD \tilde u = -\tilde q$. The latter implies~$\Vert \tilde u \Vert_V \le C_\text{invTr} \Vert \tilde q\Vert_{\Q^*}$, which leads to the estimate 
\[
  \sup_{\big[ \substack{u\\ p} \big] \in \V} 
  \frac{\big\langle \calB \big[ \substack{u\\ p} \big],\, q \big\rangle}{\Vert \big[ \substack{u\\ p} \big] \Vert_\V }
  \ge \frac{\langle \tilde p - \calD \tilde u,\, q \rangle}{\Vert \big[ \substack{\tilde u\\ \tilde p} \big] \Vert_\V}
  = \frac{\langle \tilde q,\, q \rangle}{\Vert \tilde u \Vert_V}
  = \frac{\Vert \tilde q \Vert^2_{H^{1/2}(\Gamma)}}{\Vert \tilde u \Vert_{H^1(\Omega)}}  
  \ge \frac{\Vert \tilde q \Vert_{\Q^*}}{C_\text{invTr}} 
  = \frac{1}{C_\text{invTr}}.
\]
The remaining part of the proof is similar to the one of~Lemma~\ref{lem:infsup:local}, using the fact that the Laplace-Beltrami operator satisfies a G\aa{}rding inequality on~$H^1(\Gamma)$, cf.~\cite[Ch.~16.1]{GilT01}. 
\end{proof}
As a direct consequence of Lemma~\ref{lem:infsup:nonLocal} we conclude the unique solvability of the PDAE~\eqref{eq:PDAE:nonlocal}. 
\begin{theorem}
Given a bounded Lipschitz domain~$\Omega$, a right-hand side~$\big[ \substack{f\\ g} \big] \in L^2(0,T;\V^*)$, and consistent initial data~$\big[ \substack{\uo\\ \po} \big] \in \V_0$, there exists a unique distributional solution of~\eqref{eq:PDAE:nonlocal} with~$\big[ \substack{u\\ p} \big] \in L^2(0,T; \V)$, i.e, $u \in L^2(0,T; H^1(\Omega))$~and~$p \in L^2(0,T; H^{1}(\Gamma))$.  
\end{theorem}
%
%
\section{Conclusion and discussion}
We have presented an alternative abstract formulation of parabolic systems with dynamic boundary conditions, namely as constrained operator system. This novel approach aims to provide a new perspective on these type of systems. In particular, this is of interest for the construction of new robust discretization methods. Since we include the coupling of~$u$ and~$p$ in form of an equation rather than restricting the ansatz space as in~\cite{KovL17}, we allow different spatial discretizations in the bulk and on the boundary. Further, the obtained structure in~\eqref{eq:PDAE:local} and~\eqref{eq:PDAE:nonlocal} enables a direct application of algebraically stable Runge-Kutta schemes~\cite{AltZ18}. 
%
\section*{Acknowledgement}
The author thanks David Hipp (KIT Karlsruhe) and Christoph Zimmer (TU Berlin) for valuable discussions on all sorts of boundary conditions.
%

%
\end{document}